\documentclass[a4paper,leqno]{article}
\usepackage[T1]{fontenc}

\usepackage{amsfonts}
\usepackage{amssymb}
\usepackage{amsmath}
\usepackage{amsthm}
\usepackage{mathrsfs}
\usepackage[all]{xy}
\usepackage{graphicx}

\usepackage{hyperref}
\hypersetup{
    colorlinks=true,
    breaklinks,
    linkcolor=blue,
    filecolor=magenta,
    urlcolor=blue,
    citecolor=magenta,
    linktoc=all,
}
\usepackage{cleveref}
\usepackage{aliascnt}

\usepackage{enumitem}

\usepackage{comment}

\newtheorem{thmA}{Theorem}

\newcounter{onlyone}
\numberwithin{onlyone}{section}

\newtheorem{thm}[onlyone]{Theorem}

\newtheorem{lem}[onlyone]{Lemma}
\theoremstyle{definition}
\newtheorem{defi}[onlyone]{Definition}
\newtheorem{block}[onlyone]{}

\newtheorem{rem}[onlyone]{Remark}
\newtheorem{ex}[onlyone]{Example}
\newtheorem{nota}[onlyone]{Notation}

\AtBeginEnvironment{figure}{\refstepcounter{onlyone}}

\AtBeginEnvironment{equation}{\refstepcounter{onlyone}}

\usepackage{pifont}

\graphicspath{{./}{figs/}}

\newcommand{\C}{\mathbb{C}}
\newcommand{\Z}{\mathbb{Z}}

\newcommand{\A}{\mathcal{A}}
\newcommand{\CP}{\mathbb{CP}}
\newcommand{\Strm}[2]
{\mathop{\mathrm{Str}}\nolimits^\circleddash\left(#1\,\middle|\,#2\right)}
\newcommand{\bigconnsum}{\mathop{\#}}

\newcommand{\set}[2]{\left\{ #1 \,\middle\vert\, #2 \right\}}

\newcounter{dummy}
\renewcommand{\thedummy}{\roman{dummy}}
\crefformat{dummy}{#2(#1)#3}

\newenvironment{blist}
{
  \begin{list}{(\thedummy)}
  {
    \setlength\labelsep{4pt}
    \setlength\itemindent{4pt}
    \setlength\leftmargin{0pt}
    \setlength\labelwidth{0pt}
    \setlength\parsep{0pt}
    \usecounter{dummy}
  }
}
{
  \end{list}
}

\newcommand{\GNSz}{G_{\mathrm{NSz}}}
\newcommand{\Gn}{G_{\mathrm{Neu}}}
\newcommand{\Dc}{{\bf D}_c}
\newcommand{\Dd}{{\bf D}_d}
\newcommand{\Dz}{{\bf D}_0}
\newcommand{\Cc}{\mathscr{C}}
\newcommand{\GCc}{\Gamma_\mathscr{C}}

\newcommand{\SPen}{\mathcal{N}}
\newcommand{\Pen}{\mathcal{P}}
\newcommand{\DPen}{\mathcal{D}}

\newcommand{\Nreg}{N_{\mathrm{reg}}}
\newcommand{\Nsp}{N_{\mathrm{sp}}}
\newcommand{\Ndreg}{\mathscr{N}_{\mathrm{reg}}}
\newcommand{\Ndsp}{\mathscr{N}_{\mathrm{sp}}}

\newcommand{\Nd}{\mathscr{Nd}}
\newcommand{\V}{\mathscr{V}}
\newcommand{\W}{\mathscr{W}}
\newcommand{\Ar}{\mathscr{A}}
\newcommand{\VC}{\mathscr{V_{\mathscr{C}}}}

\newcommand{\VNSz}{\mathscr{V_{\mathrm{NSz}}}}
\newcommand{\WNSz}{\mathscr{W_{\mathrm{NSz}}}}
\newcommand{\ArNSz}{\mathscr{A}_{\mathrm{NSz}}}

\newcommand{\vNSz}{v^{\mathrm{NSz}}}
\newcommand{\wNSz}{w^{\mathrm{NSz}}}
\newcommand{\arNSz}{a^{\mathrm{NSz}}}

\newcommand{\eNSz}{e^{\mathrm{NSz}}}
\newcommand{\gNSz}{g^{\mathrm{NSz}}}

\newcommand{\tX}{\tilde X}

\newcommand{\tS}{\tilde \Sigma}

\newcommand{\mleq}{<\hspace{-.15cm}<}

\title{The Milnor fiber boundary of an arrangement determines its combinatorics}

\author{
Baldur Sigur{\dh}sson\footnote{
Universidad Polit\'ecnica Madrid,
Dpto. Matem\'atica e Inform\'atica Aplicadas a las Ingenier\'ias Civil y Naval,
C. del Profesor Aranguren 3, 28040 Madrid.
\href{mailto:baldursigurds@gmail.com}{baldursigurds@gmail.com}
Spanish grant of MCIN
(PID2020-114750GB-C32/AEI/10.13039/501100011033)}
\and
Juan Viu-Sos\footnote{
Universidad Polit\'ecnica Madrid,
Dpto. Matem\'atica e Inform\'atica Aplicadas a las Ingenier\'ias Civil y Naval,
C. del Profesor Aranguren 3, 28040 Madrid.
\href{mailto:juan.viu.sos@upm.es}{juan.viu.sos@upm.es}. Spanish grant of MCIN
(PID2020-114750GB-C32/AEI/10.13039/501100011033).
}
}

\begin{document}
\maketitle

\begin{abstract}
The boundary of the Milnor fiber associated with a complex
line arrangement is a three dimensional plumbed manifold,
and it is a combinatorial invariant.
We prove the reverse implication, which was conjectured
N\'emethi and Szil\'ard.
That is, this boundary of the Milnor fiber
determines the combinatorics of the arrangement.
Furthermore, we give an explicit method which constructs the poset associated 
with the arrangement, given a plumbing graph in normal form for the boundary.
\end{abstract}

\tableofcontents

\section{Introduction}
\begin{block}
A (projective) line arrangement $\A = \{\ell_1,\ldots,\ell_d\}$ is a finite set of (distinct) lines in the complex projective plane $\CP^2$. These objects are studied as both a family of (reducible) plane curves and a particular non-isolated homogeneous surface singularity $(X_\A,0)\subset\C^3$: their interest lies in the interplay between combinatorics, topology, and algebraic geometry. For general background on ($n$-dimensional hyperplane) arrangements, see~\cite{OrlikTerao92,Dimca:book}. %
A central goal is to understand the topological and geometric properties of arrangements through their \emph{combinatorics}, namely the intersection poset $P_\A$ determined by the lines and their intersection points.

It is known that the \emph{(embedded) topology of $\A$}, i.e. the homeomorphism type of the pair $(\CP^2,\A)$, determines the combinatorics. However, the converse was proved to be false, see e.g.~\cite{Rybnikov,ACCM:real_ZP,GueViu:config}. Some classic topological invariants of the complement $U_\A:=\CP^2\setminus\bigcup_{i=1}^d \ell_i$ turn out to be combinatorial, such as its cohomology ring $H^\ast(U_\A;\Z)$~\cite{OrlikSolomon}), whereas the fundamental group is not~\cite{Rybnikov,GBVS:real_pi1}. %
The combinatorial study of other properties as the topology of the Milnor fiber associated to $\A$, twisted cohomologies, characteristic varieties or finer invariants of $\pi_1(U_\A)$ has become an active interesting domain, see e.g.~\cite{Suciu17:survey,Dimca:book}.

The Milnor fiber $F_\A$ of $\A$ (i.e. the Milnor fiber of the surface germ $(X_\A,0)$) has been extensively studied so far~\cite{Suciu14:survey,Suciu17:survey,Dimca:book,PapadimaSuciu:survey}. Nevertheless, it remains unknown whether the first Betti number of $F_\A$ is combinatorial or not. On the other hand, the boundary $\partial F_\A$ is known to be a \textit{plumbed} three manifold. In \cite{Nem_Szil}, N{\'e}methi and Szil\'ard gave an explicit algorithm which produces a plumbing graph for this manifold, in a greater generality than considered in this text. %
In~\cite[6.1]{Nem_Szil}, the authors applied their main algorithm to construct the plumbing graph in the case of arrangements, constructing $\partial F_\A$ from the combinatorics of $\A$ (see also the discussion in~\cite[24.3.1]{Nem_Szil}). In particular, they derived several (combinatorial) formulas of topological invariants of $\partial F_\A$ and its monodromy operator. This point of view was recently used by Sugawara~\cite{sugawara1}  in order to obtain a nice combinatorial formula of the torsion in $H_1(\partial F_\A;\Z)$ for generic arrangements. %

After $\partial F_\A$ was combinatorially determined by N\'emethi and Szil\'ard, they conjectured that the reverse implication should also hold~\cite[24.4.11]{Nem_Szil}: namely, that the three dimensional manifold $\partial F_\A$ also determines the combinatorics 
of $\A$. %
In this paper, we give an affirmative answer to this conjecture. More concretely, we prove the following result.
\end{block}

\begin{thmA}\label{thm:intro}
The boundary of the Milnor fiber associated with a line arrangement
determines the combinatorics of $\A$ by an explicit algorithm.
\end{thmA}

\begin{rem}
By \emph{explicit algorithm} we mean the following. Since every $\partial F_\A$ is a plumbed three manifold, one can associate a plumbing graph $G$ which realizes the diffeomorphism type $\partial F_\A\cong M(G)$. Using \textit{plumbing calculus}~\cite{Neu_plumb}, the \textit{normal form} $\Gn$ of the plumbing graph is algorithmically obtained. Finally, we are able to reconstruct $P_\A$ from $\Gn$, identifying the combinatorics. 
This algorithm can be applied to any plumbing graph, to see if
it can be constructed from a suitable poset. %

It says nothing,
however, about the realizability of this poset as the combinatorics of
an arrangement. %
One should know \emph{a priori} that $\partial F$ comes from an arrangement $\A$ in order to give sense to the output poset. %
As an example, let $P$ be a poset obtained as follows: take the poset of the \emph{Pappus arrangement} (i.e. the one with 9 lines and 9 triple points verifying the Pappus' theorem) and remove the relations between the line and the three triple points appearing in the conclusion of the theorem. %
From $P$ we can construct a plumbed manifold, but this would not be diffeomorphic to some $\partial F_\A$ for any $\A$, since there is no arrangement realizing $P$ as intersection poset.\\
\end{rem}

\begin{block}
This paper is organized as follows. %
In \Cref{s:arrangements}, the basic constructions about line arrangements are quickly introduced, as well as the notion of \emph{exceptional} arrangements. %
In \Cref{s:statement}, we present the main theorem together with the necessary definitions and concepts of plumbing manifolds involved. %
In \Cref{s:nem_szil} we go through the algorithm of N\'emethi and
Szil\'ard in the case of an arrangement and describe the output
$\GNSz$. We also introduce notation
which will be used throughout the rest of the paper. %
In \Cref{s:pencils} we recall known results about plumbing manifolds of pencils and
near-pencils. %
In \Cref{s:G}, we apply some of the steps of the process of normalizing
the plumbing graph $\GNSz$ to produce a graph $G$, which is in normal
form for \emph{non-exceptional} arrangements. %
In \Cref{s:double} we describe a plumbing graph in normal form
for $F_\A$ in the case of a double connected pencil.
This requires a careful analysis of one string in the graph $G$
introduced in the latter. %
In \Cref{s:nonexc} we show that for non-exceptional arrangements,
the boundary $\partial F_\A$ determines the poset $P_\A$.
\end{block}

\begin{nota}
We use the words \emph{string} and \emph{chain} interchangeably, the
former being used in \cite{Nem_Szil} and the latter in \cite{Neu_plumb}.
\end{nota}

\section{Preliminaries on line arrangements}\label{s:arrangements}

\begin{block}
  To any line arrangement $\A = \{\ell_1,\ldots,\ell_d\}$
  in $\CP^2$, one associates its \emph{intersection poset} $(P_\A, <)$ whose elements are:
  the projective plane $\CP^2$, 
  the lines $\ell_1,\ldots,\ell_d$, and
  singular points $p = \ell \cap \ell'$ for $\ell\neq \ell'$ in $\A$. 
  The set of these elements is a partially ordered set with respect to
  inverse inclusion, such that the plane is the unique minimal element.
  This poset has a \emph{rank}, which is given by codimension, that is,
  the plane has rank $0$, a line has rank $1$, and a point has rank $2$.
\end{block}
\begin{block}
We will assume throughout this paper that $\A$ is an arrangement of
$d$ lines. They are indexed as $\ell_1, \ldots, \ell_d$, having
precisely $c$ total intersection points, $p_{d+1}, \ldots, p_{d+c}$.
We will stick to using the letter $i$ to index objects
associated with lines, and $j$ to index objects associated with
intersection points.
In particular,
we associate with each line $\ell = \ell_i$ and point $p = p_j$
the numbers $\bar{n}_i$ of intersetion points on $\ell_i$, and $n_j$
the set of lines passing through $p_j$. Set also, for each
intersection point $p_j$, the integer
\[
  c_j = \gcd(d,n_j).
\]
\end{block}

\begin{block}\label{block:Xarrangements}
  Let $\alpha_i\in\C[x,y,z]$ be a linear function associated to each $\ell_i\in\A$. The product $f = \alpha_1\alpha_2\ldots \alpha_d$ is a defining polynomial of the underlying curve of $\A$ in $\CP^2$. Let $X_\A$ be the germ in $(\C^3,0)$ defined by $f=0$, i.e. the germ associated to the cone of $\A$ in $\C^3$. Note that $(X_\A,0)$ is the germ of a non-isolated singular surface whenever $d\geq 2$. %
  The \emph{Milnor fiber} of $X_\A$, i.e. 
  \[
    F_\A
    =
    f^{-1}(t) \cap B_\epsilon^6
    =
    \set{x \in \C^3}{f(x) = t,\;\|x\|\leq \varepsilon},
  \]
  is a well defined 4-manifold with boundary~\cite{Milnor_hyp}, assuming we
  choose parameters $0 < \eta \mleq \varepsilon \mleq 1$, and $t \in \C$
  with $0 < |t| < \eta$. We are interested in the boundary $\partial F_\A$ of the Milnor fiber, which
  is a (smooth) plumbed $3$-manifold and it will be treated in~\Cref{s:nem_szil}. Alternatively, since $f$ is a homogeneous polynomial, the Milnor fiber is also diffeomorphic to $F_A\cong f^{-1}(1)$ and $\partial F_A\cong f^{-1}(1)\cap S^5_R$ for sufficiently big radious $R\gg0$.
\end{block}

\begin{block}
  We conclude by introducing the class of line arrangements in $\CP^2$ for which the procedure described in \Cref{s:nonexc} does not work in general: exceptional arrangements. %
  We say that an arrangement $\A$ is \emph{exceptional} if there exists $\ell\in\A$ containing strictly less than three intersection points. %
  Exceptional arrangements can be completely classified in four simple families of arrangements.
  \begin{lem}
  The following list contains all exceptional line arrangements.
  \begin{enumerate}
  \item
  An arrangement of one line.
  \item
  A \emph{pencil} $\Pen_d$ of $d\geq 2$ lines passing through the same point.
  \item
  A \emph{near-pencil} $\SPen_d$ of $d$ lines, precisely $d-1$ of which
  pass through the same point.
  \item
  A \emph{double connected pencil} $\DPen_{a,b}$ of $d=a+b-1$ lines,
  $a$ of which contain a point $p_1$ and $b$ of which contain a different
  point $p_2$.
  \end{enumerate}
  \end{lem}

  \begin{proof}
  Let $\A$ be an arrangement of lines in $\CP^2$. If $\A$ contains a line
  not containing any intersection point, then the arrangement consist of
  that single line, since any other line would intersect it.

  If $\A$ contains a line $\ell$ containing only one intersection point $p$,
  then the arrangement is a pencil $\Pen_d$ with $d \geq 2$.
  Indeed, any line in $\A$ which does not contain $p$ would intersect
  $\ell$ in some point other than $p$.

  Finally, assume that $\A$ contains a line $\ell$ containing precisely two
  intersection points, say $p_1,p_2$.
  In this case, any line in $\A$ contains either $p_1$ or $p_2$,
  since otherwise it would intersect $\ell$ at some point other than
  $p_1$ or $p_2$.
  Denote by $a,b$ the valencies of $p_1,p_2$. We can assume that
  $a \geq b \geq 2$. If $b = 2$, then $\A$ is a near-pencil $\SPen_{a+1}$,
  and if $b \geq 3$, then $\A$ is a connected double pencil $\DPen_{a,b}$.
  \end{proof}
\end{block}

\section{Plumbing manifolds and main theorem} \label{s:statement}

\begin{block}
  An \emph{plumbing graph} is a finite decorated graph $G$ with the following data:
  \begin{itemize}
    \item Each vertex $v$ is weighted by a \emph{genus} $g_v \in \Z_{\ge 0}$ and a \emph{Euler number} $e_v\in \Z$.
    \item Each edge brings a signature, $+$ or $-$.
  \end{itemize}
  The genus (resp. signature) is usually written as $[g_v]$ under the vertex (resp. over the edge) (resp. as $\varepsilon$) and is omitted if it is zero (resp. if it is $+$). The Euler number $e_v$ is placed over its the vertex $v$. 

  The (oriented) \emph{plumbing manifold} $M(G)$ associated to an (oriented) plumbing graph $G$ is the closed graph 3-manifold obtained as follows:
  \begin{enumerate}
    \item To each vertex $v$, one assigns an oriented $S^1$-bundle $\pi_{v}\colon E_v \to S_v$ over orientable surface $S_v$ of genus $g_v$, and Euler obstruction $e_v$. 
    \item For any edge $(v,w)$ of signature $\varepsilon \in \{\pm\}$, choose a point $p \in S_v$ and trivialize the fibration over a small disk $D_p$ centered at $p$ as $\pi_v^{-1}(D_p) \cong D_p \times S^1$. Similarly, choose $q \in S_w$ and trivialize its bundle over a neighborhood $D_q$ as $\pi_w^{-1}(D_q) \cong D_q \times S^1$. Then glue $\pi_v^{-1}(S_v \setminus D_p)$ and $\pi_w^{-1}(S_w \setminus D_q)$ along a homeomorphism $\varphi : \partial D_p \times S^1 \to \partial D_q \times S^1$ represented by $\varepsilon\cdot\begin{pmatrix} 0 & 1 \\ 1 & 0 \end{pmatrix}$.
  \end{enumerate}
  Disconnected plumbing graphs are allowed as well as disjoint unions of graphs, denoted by the symbol $\#$ (see e.g.  fig.~\ref{fig:sdpen}).
  
  The (oriented) diffeomorphism type of $M(G)$ is preserved by \emph{(Neumann's plumbing calculus)}, a set of graph operations R0-R8 on $G$. Plumbing calculus is sufficient in the sense that any two plumbing graphs representing the same (oriented) diffeomorphism type of plumbed manifolds are related by a finite sequence of these operations. For brevity, we omit the explicit descriptions of these operations. See~\cite[4.2]{Nem_Szil} or~\cite{Neu_plumb} for more details and a complete list of operations. Throughout this text, we explicitly use only the following operations:  R0 (reversing signs), R1 (blowing down), R3 ($0$-chain absorption), R5 (oriented handle absorption) and R6 (splitting).
  
Given a plumbing graph $G$ let us
add a few vertices to the graph, which are referred to as \emph{arrow heads},
which are joined by an edge to exactly one vertex in $G$ each.
To each vertex $w$ in this bigger graph, associate an integer $m_w$.
This family is a \emph{multiplicity system} if it satisfies the equation
\cite[eq. 4.1.5]{Nem_Szil}
\begin{equation} \label{eq:multiplicity}
  e_v m_v + \sum_w \varepsilon_{vw} m_w = 0
\end{equation}
for any vertex $v$ in $G$,
where the sum runs through all neighbors $w$ of $v$, including arrowheads.
A typical example is when $G$ is a resolution graph, the multiplicities
are valuations associated with exceptional divisor of some function,
and the arrowheads represent components of the strict transform
of the set defined by the function.
Such multiplicity systems can be used to calculate Euler numbers,
as in \cite{Nem_Szil}.
\end{block}

\begin{block} \label{block:special_node}
Let $G$ be any plumbing graph.
Denote by $\delta_v$ the valency of the vertex, i.e. the number of
neighbors of $v$.
A \emph{regular node} of $G$ is any vertex having $\delta_v \geq 3$ or
$g_v \neq 0$. The \emph{regular node graph} associated with $G$
is the graph $\Nreg$, whose vertices are regular nodes in $G$,
and edges correspond to strings connecting nodes in $G$.
Denote by $\Ndreg$ the set of regular nodes.

We define a \emph{special node} to be a vertex having
precisely two neighbors, each having genus $0$ and Euler number $-1$.
The \emph{special node graph} $\Nsp$ associated with $G$ is the graph whose
vertices are regular or special nodes in $G$ and whose edges correspond to
strings in $G$. Denote by $\Ndsp$ the set of nodes, regular or special.
\end{block}

\begin{rem}
\begin{blist}
\item
If the graph $G$ is in normal form, then the two neighbors of a
special nodes do not lie on a string, i.e. they have valency $\geq 0$.
Indeed, any vertex in a graph in normal has Euler number $\leq -2$.
\item
The graphs $\Nreg$ and $\Nsp$ have the same topological type as $G$.
\end{blist}
\end{rem}

\begin{block}
Any finite poset $(P,<)$ of rank $2$ with a unique minimal element determines
a bipartite graph $B$ whose vertices are elements of rank $1$ or $2$,
with an edge joining two elements if they are comparable.
This bipartite graph comes with a chosen maximal independent set
$\Nd_1$, the set of vertices corresponding elements of rank $1$.
This correspondence is invertible, in that,
a bipartite set with a chosen maximal independent set $(B,\Nd_1)$ naturally
defines a poset, such that these two operations are inverse to each other.
We say that $(P,<)$ and $(B,\Nd_1)$ \emph{correspond to each other}.
\end{block}

\begin{block}
A connected plumbing graph $G$ \emph{determines} a poset $(P,<)$
if the special node graph $\Nsp$ associated with $G$ with vertex set $\Ndsp$
is bipartite, and
has a unique partitioning into maximal independent sets
$\Ndsp = \Nd_1 \amalg \Nd_2$ such that
\begin{itemize}
\item
$(\Nsp, \Nd_1)$ corresponds to $(P,<)$.
\item
nodes in $\Nd_1$ have Euler number $-1$ and genus $0$,
\item
$\Nd_1$ contains no special nodes,
\item
if a chain in $G$ joining $v \in \Nd_1$ and $w \in \Nd_2$
has Euler numbers $-h_1,\ldots,-h_r$, from $v$ to $w$,
and $w$ is a regular node, then
\[
  [h_1,\ldots,h_r] = \frac{|\Nd_1|}{\delta_w}.
\]
\end{itemize}
\end{block}

It is worth noticing that not every plumbing graph determines a poset, e.g. complete bipartite graphs. The main strategy of our classification is that any normalized plumbing graph of a line arrangement determines a poset, except for finitely many families of arrangements. Namely, these are the trivial arrangement, (near-)pencils and double connected pencils, i.e. the four families of exceptional arrangements. For each of these, their associated $\Gn$ is also classified. %

We now present our main result, which is a reformulation of Theorem~\ref{thm:intro}, since obtaining $\Gn$ from a plumbing graph $G$ is an algorithm process (see \Cref{s:nem_szil}). The proof is organized across the following sections.

\begin{thm} \label{thm:main}
Let $\Gn$ be a plumbing graph in normal form
representing the boundary $\partial F_\A$
of the Milnor fiber associated with an arrangement $\A$ of lines in
the complex projective plane $\CP^2$.
\begin{enumerate}[label=(\roman*)]

\item \label{it:main_empty}
$\A$ contains precisely one line
if and only if
$\Gn$ is the empty graph.

\item \label{it:main_pencil}
$\A$ is a pencil $\Pen_d$ on $d \geq 2$ lines
if and only if
$\Gn$ is the disjoint union of $(d-1)^2$ vertices with genus and Euler
number zero.

\item \label{it:main_semi}
$\A$ is a near-pencil $\SPen_d$ on $d \geq 3$ lines
if and only if
$\Gn$ consists of one vertex with Euler number $0$ and genus $d-2$,

\item \label{it:main_double}
$\A$ is a double connected pencil $\DPen_{a,b}$ with $a\geq b \geq 3$
if and only if
the regular node graph of $\Gn$ is a complete bipartite graph on
$a$ and $b$ vertices.

\item \label{it:main_nonexc}
$\A$ is non-exceptional if and only $\Gn$ determines
a poset $(P,<)$ of rank $2$, and in this case, $(P,<)$ is the poset
associated with $\A$.

\end{enumerate}
\end{thm}

\begin{proof}
A plumbing graph in normal form can only belong to one of the
categories described in
\ref{it:main_empty},
\ref{it:main_pencil} or
\ref{it:main_semi}.
Furthermore, none of these graphs are complete bipartite graphs on
$a\geq b \geq 3$ vertices, nor do they determine a poset of rank $2$.
Finally, the categories
\ref{it:main_double},
\ref{it:main_nonexc}
do not intersect, by Lemma \ref{lem:not_bipartite}.
As a result, the theorem follows from
\ref{block:pencil} and \ref{block:semipencil},
and Lemmas \ref{lem:double_pencil}, \ref{lem:nonexc}
and \ref{lem:not_bipartite}.
\end{proof}

\begin{figure}[ht]
\begin{center}
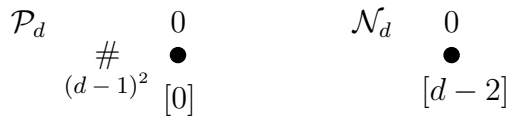
\caption{The graph $\Gn$ in the case of pencils and near-pencils.}
\label{fig:sdpen}
\end{center}
\end{figure}

\begin{rem}
Although one might consider the near-pencil $\SPen_d$ as a
connected double pencil $\DPen_{a,b}$ with parameters $a=d-1$ and $b=2$,
Neumann's normalization algorithm must be carried out further for
the near-pencil than
the steps described in \Cref{s:double}. In this respect, the two cases
are essentially different.
\end{rem}

\section{The algorithm of N\'emethi and Szilárd} \label{s:nem_szil}

\begin{block}
Take $f=\prod_{\ell_i\in\A}\alpha_i$ the homogeneous polynomial defining the line arrangement $\A$ as in \ref{block:Xarrangements}, where each linear form $\alpha_i$ defines the line $\ell_i$. Denote by $X_i \subset \C^3$ the hyperplane
defined by $\alpha_i = 0$. One has that $X_\A = \cup_i X_i$.
For each point $p_j$, denote by $\Sigma_j$ the corresponding line in $\C^3$,
and by $\Sigma$ the union of these lines. This way, $(X_\A,0)$ is the germ
of a non-isolated hypersurface singularity, with
one dimensional singular locus $(\Sigma,0)$.
By \cite{Nem_Szil}, the boundary of the Milnor fiber
\begin{equation} \label{eq:boundary}
  \partial F = f^{-1}(\eta) \cap S^5_\epsilon
\end{equation}
is a plumbed 3-manifold, represented by a plumbing graph, which we
call $\GNSz$.
Denote by $g \in \C[x,y,z]$ a linear function, defining a line
$\ell_0$ in the plane
which does not pass through any of the singular points $p_1,\ldots, p_c$ of $\A$, and
let $Y$ be the hyperplane in $\C^3$ defined by $g=0$.
This way, $f,g$ define an isolated complete intersection consisting of
$d$ lines in $\C^3$. 
Blowing up the origin in $\C^3$ separates the lines $\Sigma_j$. Denote by
\[
  r:V\to\C^3
\]
the map obtained by first blowing up the origin in $\C^3$, and then
blowing up each of the strict transforms of $\Sigma_j$. This second
step is taken in order to satisfy \emph{Assumption A}, described
in \cite[6.3]{Nem_Szil}.
Setting
\[
  \Dc = \overline{ r^{-1}(X_\A \setminus Y) },\qquad
  \Dd = \overline{ r^{-1}(Y \setminus X_\A) },\qquad
  \Dz = r^{-1}(X_\A \cap Y),
\]
the \emph{curve configuration} is
\[
  \Cc = (\Dc \cap \Dz) \cup (\Dc \cap \Dd).
\]
This is a union of curves in $V$, whose dual graph is denoted
$\GCc$. The set of vertices of this graph is $\VC$, with $v \in \VC$
corresponding to a component of $\Cc$. We write $\V = \W \amalg \Ar$,
with $v \in \W$ if and only if the corresponding component is compact.
The vertices in $\W$ are drawn as regular dots, whereas the ones in $\Ar$
are drawn as arrowheads.
Two vertices in $\GCc$ are joined by an edge if the corresponding
components intersect.
\begin{itemize}
\item
For each of the lines $\ell_i$, denote by $\tX_i$ the strict transform
of $X_i$ in $V$. Thus, we have a vertex $v^\Cc_i$ in $\GCc$ corresponding
to the curve $\tX_i \cap r^{-1}(0)$.
\item
For each of the points $p_j$, denote by $\tS_j$ the strict transform
$\overline{ r^{-1}(\Sigma_j \setminus \{0\}) }$ of $\Sigma_j$ in $V$.
Thus, we have a vertex $w_j^\Cc \in \W$ in $\GCc$, corresponding to the
compact curve $\tS_j \cap r^{-1}(0)$.
\item
The vertices $v^\Cc_i$ and $w^\Cc_j$ are joined by an edge if and only
if $p_j \in \ell_i$.
\item
Every line $\ell_i$ intersects the line $\ell_0$ in a single point,
corresponding to a line in $\C^3$. The strict transform of this line
corresponds to an arrowhead vertex $a^\Cc_i$ in $\GCc$, which
is joined by an edge with $v^\Cc_i$.
\end{itemize}
Furthermore, the graph $\GCc$ is decorated with the following
numerical data
\begin{itemize}
\item
The vertex $v^\Cc_i$ is decorated by $(1,d,1)$.
\item
The vertex $w^\Cc_j$ is decorated by $(n_j,d,1)$.
\item
If $p_j \in \ell_i$, then the corresponding edge joining the vertices
$v^\Cc_i$ and $w^\Cc_j$ is decorated by $2$.
\item
An arrowhead $a^\Cc_i$ is decorated by $(1,0,1)$, and the edge joining
it with $v^\Cc_i$ is decorated by $1$.
\item
All compact curves in $\Cc$ have genus zero, so we decorate the
corresponding vertices with $[0]$.
\end{itemize}
\end{block}

\begin{figure}[ht]
\begin{center}
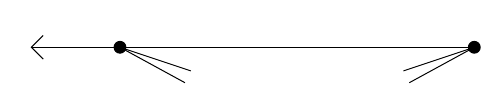
\caption{Part of the graph $\GCc$ with decorations.}
\label{fig:conf_graph}
\end{center}
\end{figure}

\begin{block}
The \emph{Main Algorithm} of \cite[Chapter 10]{Nem_Szil} takes as input the graph
$\GCc$, and returns a plumbing graph for the boundary of the
Milnor fiber \cref{eq:boundary}.
We start by making two remarks:
\begin{itemize}
\item
The middle decorations on any non-arrowhead vertex is $d \neq 0$.
As a result, \emph{Assumption B} \cite[10.1.3]{Nem_Szil} is satisfied.
Assumption A is also satisfied, as we have seen, and so
no further modification of the the resolution $r$ is needed
before starting the Main Algorithm.
\item
The covering data in step two of the Main Algorithm
\cite[10.2.5]{Nem_Szil}
is trivial in all cases, i.e. we have
$\mathfrak{n}_w = \mathfrak{n}_e = 1$ for any vertex $w$ and edge $e$
in the graph $\GCc$. As a result, we present here the output of the\
Main Algorithm without mentioning these data.
\end{itemize}
Denote the output of the Main Algorithm by $\GNSz$.
This is a plumbing graph with a multiplicity system. 
More precisely, $\GNSz$ is a graph with vertex set
$\VNSz = \WNSz \amalg \ArNSz$, where $\ArNSz$ are arrowheads.
Vertices of the form $a^\Cc_i$, $v^\Cc_i$, $w^\Cc_j$ induce
similar ones $\arNSz_i$, $\vNSz_i$, $\wNSz_j$ in the graph $\GNSz$.
The non-arrowhead vertices are decorated with Euler numbers
$\eNSz_i$ and $\eNSz_j$ and genera $\gNSz_i$ and $\gNSz_j$.
Following the Main Algorithm, one obtains formulas for the genera
\[
  \gNSz_i = 0,\qquad
  \gNSz_j = \frac{(c_j - 1)(n_j - 1)}{2},
\]
associated with any line $\ell_i$ and any intersection point $p_j$.
An arrow-head has multiplicity $1$,
a vertex $\vNSz_i$ also has multiplicity $1$,
and a vertex $\wNSz_j$ has multiplicity
\[
  m_j = \frac{n_j}{c_j}.
\]

Any edge labelled with $1$ joining $a^\Cc_i$ and $v^\Cc_i$ induces
an edge labelled with $+1$ joining $\arNSz_i$ and $\vNSz_i$.
Any edge labelled with $2$ joining vertices $v^\Cc_i$ and $w^\Cc_j$
induces a string of type $\Strm{1,n_j;d}{0,0;1}$.
This is a string with Euler numbers $k_1,\ldots,k_s \geq 2$
and multiplicities $m_1,\ldots,m_s$.
The numbers $k_1,\ldots,k_s$ are determined as a
\emph{negative continued fraction expansion}
\[
  \frac{d}{\lambda}
  =
  [k_1,k_2,\ldots,k_s]
  =
  k_1 - \frac{1}{k_2 - \frac{1}{\ddots - \frac{1}{k_s}}}
\]
where $0 \leq \lambda < d/c_j$ is determined by
\[
  n_j + \lambda = m_1 d, \qquad m_1 \in \Z,
\]
in other words, $m = 1$ and $\lambda = d-n_j$. Therefore, we have
\[
  \frac{d}{d-n_j}
  =
  [k_1,k_2,\ldots,k_s],
\]
with the exception when $d = n_j$, in which case we join the two vertices
with a single negative edge.
Furthermore, if we set $m_0 = 1$ and $m_{s+1} = m_j$, the multiplicities
satisfy
\[
  k_l m_l - m_{l-1} - m_{l+1} = 0,\qquad l=1,\ldots,s.
\]
\end{block}

\begin{figure}[ht]
\begin{center}
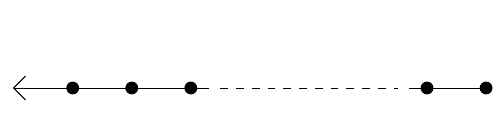
\caption{The string $\Strm{1,n_j;d}{0,0;1}$ in $\GNSz$.}
\label{fig:strn}
\end{center}
\end{figure}

\begin{defi}
For every intersection point $p_j$, with the notation introduced above,
set $m'_j = m_s$, the multiplicity on any neighbor of $w_j$ in $\GNSz$. 
\end{defi}

\begin{lem} \label{lem:eNSz}
The Euler numbers are determined by
\[
  \eNSz_i = \bar{n}_i - 1,\qquad
  \eNSz_j = m'_j c_j.
\]
\end{lem}
\begin{proof}
The vertex $\vNSz_i$ has multiplicity $1$, and $\bar{n}_i+1$ neighbors, each
with multiplicity $1$, $\bar{n}_i$ of which is joined by a negative edge.
As a result, by \cref{eq:multiplicity}, we find
\[
  \eNSz_i - \bar{n}_i + 1 = 0,
\]
proving the first equation.
Applying the same formula for the vertex $\wNSz_j$, we find
\[
  \eNSz_j m_j - n_j m'_j = 0.
\]
Since $m_j = n_j / c_j$, the result follows.
\end{proof}

\section{Normal forms of pencils and near-pencils} \label{s:pencils}

\begin{block}
In this section, we describe explicitly the normal form of plumbing graphs
for $\partial F_\A$ in the case of pencils and near-pencils.
These results can be found in \cite{Nem_Szil,sugawara1}.
\end{block}

\begin{block} \label{block:pencil}
Consider the pencil $\Pen_d$. There is a single intersection point
$p_j$ (with $j=d+1$), with $n_j = d$. As a result, we find
$c_j = d$ and $m_j = 1$, and so
\begin{equation} \label{eq:genus_pencil}
  g_j = \frac{(d-2)(d-1)}{2}.
\end{equation}
We have vertices $\vNSz_1,\ldots,\vNSz_d$ in the
graph $\GNSz$, each with Euler number and genus zero, and each
of them joined with the vertex $\wNSz_j$ by a simple edge.
By choosing one of the vertices $\vNSz_i$, we can apply R6, splitting,
to the graph $\GNSz$. Each of the components $\Gamma_1,\ldots,\Gamma_s$
in \cite[p. 305-306]{Neu_plumb}
consists of one of the vertices $\vNSz_i$, and there are $d-1$ of them. 
The numbers $k_1,\ldots,k_s$ are all 1. Using \cref{eq:genus_pencil},
we find that $\GNSz$ is equivalent to the disjoint union of
\[
  s + 2g_j = d-1 + 2\frac{(d-1)(d-2)}{2} = (d-1)^2
\]
vertices with Euler number and genus zero.
\end{block}

\begin{block} \label{block:semipencil}
Next, consider the near-pencil $\SPen_d$, with $d\geq 3$ lines.
If $d > 3$, then there is a unique line $\ell_1$ with $n_1 = d-1$
and a unique point $p_{d+1}$ with $n_{d+1} = d-1$. The rest of the
lines contain only two intersection points, and the rest of the points
are double. Thus, the graph $\GNSz$ consists of the vertices
$\vNSz_1$ and $\wNSz_{d+1}$, joined by $d-1$ bamboos of the form
seen in \cref{fig:semi_NSz}. 
\end{block}

\begin{figure}[ht]
\begin{center}
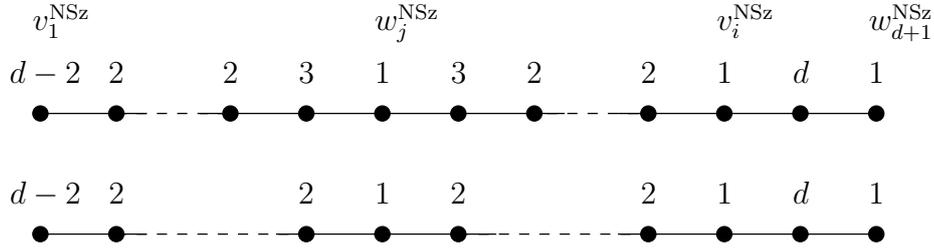
\caption{We have $d-1$ bamboos of the above form, for
$i=2,\ldots,d$, and $j = d+i$. The first string shows the case
when $d$ is odd, whereas the second shows the case when $d$ is even.
All edges in these graphs are negative.}
\label{fig:semi_NSz}
\end{center}
\end{figure}

Applying blow-downs and zero-chain absorptions to these strings,
we find the graphs seen in \cref{fig:semi_finish}.

\begin{figure}[ht]
\begin{center}
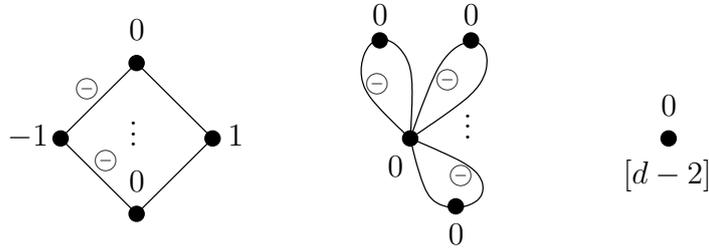
\caption{To the left, the vertices with Euler number $\pm 1$
are joined by $d-1$ bamboos, each having one vertex with Euler number
$0$, and one negative edge. Applying one zero-chain absorption, we
get a graph with $d-2$ bamboos joining a $0$-vertex with itself.
Now, apply an oriented handle absorption, R5, to each of these loops,
ending with a single vertex with Euler number $0$ and genus $d-1$.}
\label{fig:semi_finish}
\end{center}
\end{figure}

\section{An (almost) minimal graph for arrangements} \label{s:G}

\begin{block}
In this section, we define a graph $G$, associated with a line arrangement
$\A$, obtained by following certain parts of Neumann's algorithm
for finding a normal plumbing graph, which we denote by $\Gn$.
We have described the normal forms of $\partial F_\A$ in the case
of a pencil in \cref{s:pencils}, so we will assume that $\A$
is not a pencil. In particular, we have $n_j < d$ for all
intersection points $p_j$.
\end{block}

\begin{block}
We describe a plumbing graph explicitly from the poset $P_\A$.
To each line $\ell_i$ we associate a vertex $v_i$, with Euler number and
genus
\[
  e_i = -1,\qquad g_i = 0.
\]
To each intersection point $p_j$, we associate a vertex $w_j$.
If $n_j>2$, then its Euler number and genus are
\begin{equation} \label{eq:egj}
  e_j = \eNSz_j - n_j,\qquad
  g_j = \gNSz_j = \frac{(c_j-1)(n_j-2)}{2},
\end{equation}
and for double points, $n_j = 2$, we set
\[
  e_j = -d,\qquad g_j = 0.
\]
If $p_j \in \ell_i$ and $n_j > 2$,
then we join the vertices $v_i, w_j$ by a string
as in \cref{fig:Gbamboo}.
All the edges in this string are positive, and it has
negative Euler numbers $h_1,\ldots,h_r$ determined by
\[
  \frac{d}{n_j} = [h_1,h_2,\ldots,h_r].
\]
All edges in on this string have a positive sign.

\begin{figure}[ht]
\begin{center}
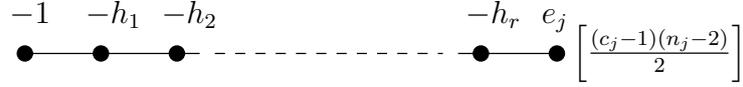
\caption{A string in the graph $G$ for a multiple point
$p_j$, $n_j > 2$. Observe that all edges have a positive sign.}
\label{fig:Gbamboo}
\end{center}
\end{figure}

If $p_j$ is a double point, the intersection point of $\ell_i$ and $\ell_{i'}$,
then we join $w_j$ with $v_i$ and $v_{i'}$ by one edge each,
one positive, and one negative. Note that it is not important
which one has which sign, as they can be swapped by applying
the operation R0 to the vertex $w_j$.
\end{block}

\begin{lem}
The graphs $\GNSz$ and $G$ are equivalent plumbing graphs, i.e. $G$ is
obtained by applying a sequence of operations R0-8 to $\GNSz$.
\end{lem}
\begin{proof}
The graph $G$ is obtained by applying the following sequence of 
operations
On one hand, on each string joining two vertices $\vNSz_i$ and $\wNSz_j$
in $\GNSz$ with $n_j > 2$:
\begin{enumerate}
\item
Perform a negative blow-up at every edge,
the inverse of R1 with $\epsilon = -1$.
This lowers the Euler number at every vertex in $\GNSz$ by its valency.
\item
Every vertex in the string which has Euler number $2$
in $\GNSz$ now has Euler number $0$. Perform a $0$-chain absorption,
 R3, on it.
\item
Any vertex on the string which has Euler number $e>2$
has Euler number $e-2$ in the resulting graph, so far. Perform
$e-3$ zero-extrusions on it, i.e. the inverse of R3, thus replacing
it with a substring with alternating Euler numbers $1$ and $0$.
\item
Blow down the vertices introduced in the last step which have Euler number $1$.
\item
Observe that all but one negative sign on each original string is deleted,
leaving precisely one, which we can take as the edge next to a vertex
$v_i$.
\end{enumerate}

\begin{figure}[ht]
\begin{center}
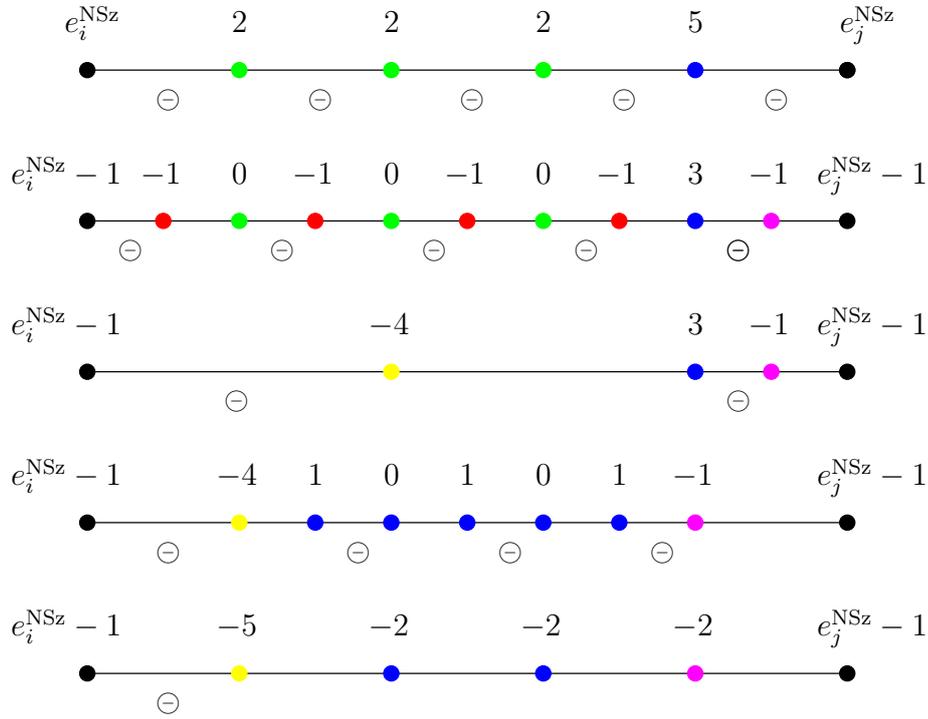
\caption{Normalization of a string with positive Euler numbers.}
\label{fig:pmstrings}
\end{center}
\end{figure}

By this procedure, a maximal subsequence $k_{l+1} = \ldots = k_{l+m} = 2$
of $2$'s of  length $m$
in $k_1,\ldots,k_s$ gets turned into a single $-m-2$ if the subsequence
is at either end, and a $-m-3$ if it is in the interior. Note that this also
applies to subsequences of length $m=0$, i.e. an edge connecting
$k_{i-1}$ and $k_{i+1}$ with both $k_i,k_{i+1} \geq 3$.
Similarly, any $k_l \geq 3$ is turned into a string of $-2$'s of length
$k_i-2$ if $i\neq 1,s$, and of length $k_i-3$ otherwise.
If the sequence of Euler numbers obtained this way is
$-h'_1,\ldots,-h'_{r'}$, then by \cite[Proposition 2.7]{P-P_cfrac}, we have
\[
  \lambda = [k_1,\ldots,k_s] \quad
  \Rightarrow \quad
  \frac{\lambda}{1-\lambda} = [h'_1, \ldots, h'_{r'}].
\]
that is,
\[
  [h'_1, \ldots, h'_{r'}]
  =
  \frac{d}{n_j}
  =
  [h_1, \ldots, h_r]
\]
and so the sequences $h'_1,\dots,h'_{r'}$ and $h_1,\ldots,h_r$ coincide.

On the other hand,
if $p_j$ is a double point in $\A$, contained in two lines $\ell_1,\ell_2$,
then the subgraph of $\GNSz$ on the vertices $\vNSz_1, \vNSz_2,\wNSz_j$
and the two strings connecting them is transformed by the
same operation to one of the strings on the left of \cref{fig:special_d},
depending on the parity of $d$.
In the even case, a zero-chain absorption yields the string on the right,
whereas in the odd case, two blow-downs and a zero-chain absorption
gives the same result.

\begin{figure}[ht]
\begin{center}
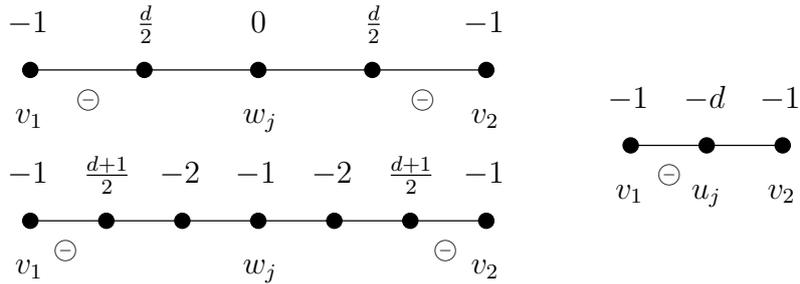
\caption{Part of normalizing a double point.}
\label{fig:special_d}
\end{center}
\end{figure}

Now, observe that for every line $\ell_i$, all these operations have
lowered the Euler number of $\vNSz_i$ by one, precisely once for each
multiple point $p_j\in \ell_i$. Since $\eNSz_i = \bar{n}_i-1$, we find
that the vertex $v_i$ in $G$ has Euler number $-1$.
By a similar argument, we have a vertex $w_j$ in $G$ for every intersection point $p_j$ with $n_j > 2$, with Euler number and genus given by
\cref{eq:egj}.

Finally, apply operation R0 to every vertex $v_i$ in the resulting
graph, and we remove the one minus on any string joining $v_i$ and $w_j$
if $n_j > 2$. Each special node $w_j$ corresponding to a double point
$p_j$ is still adjacent to one negative edge.
\end{proof}

\begin{lem} \label{lem:double_normal}
If $\A$ is a non-exceptional arrangement, then the graph
$G$ is in normal form, as in \cite[\textsection 4]{Neu_plumb}.
\end{lem}

\begin{proof}
We verify conditions N1-6 in \cite[\textsection 4]{Neu_plumb} for $G$.
By assumption, all vertices $v_i$ corresponding to lines $\ell_i$ have
valency $\delta_i \geq 3$.
Any vertex of valency $2$ in this
graph therefore either lies on a string corresponding to
$p_j \in \ell_i$ with $n_j > 2$, and so has Euler number $h_l \leq -2$,
or it is a special node corresponding to a double point, and has
Euler number $-d \leq -2$.
In particular, we have verified condition N2 for $G$.

Now, the graph $G$ has no vertex of valency $1$ or $0$.
With what we have seen so far, one verifies that the operations R1-8
cannot be applied to $G$, verifying condition N1.
Having no leaf, the graph satisfies N3-5.
Since $G$ is connected and does have a vertex of valency
$\geq 3$ (take any $v_i$) condition N6 is also satisfied.
\end{proof}

\section{Double connected pencils and complete bipartite graphs} \label{s:double}

Among the exceptional families of arrangements, the only remaining case is that of a double connected pencil: we assume that $\A = \DPen_{a,b}$ for some $a\geq b \geq 3$.

\begin{lem} \label{lem:double_pencil}
Let $\Gn$ be a plumbing graph in normal form for the boundary
$\partial F_\A$ of a Milnor fiber associated with a
connected double pencil $\DPen_{a,b}$ with $a\geq b \geq 3$,
and let $\Nreg$ be its regular node graph.
Then $\Nreg$ is a complete bipartite graph on $a$ and $b$ vertices.
\end{lem}

\begin{proof}
Let $G$ be the graph associated to $\A=\DPen_{a,b}$, following the construction of \Cref{s:G}. Assume that
$p_j$ and $p_k$ are the two centers of the pencils in $\DPen_{a,b}$,
so that $n_j = a$, $n_k = b$.
Assume also that $\ell_1$ is the line passing through $p_j,p_k$.
The node graph associated with $G$ is a complete bipartite graph
on $a,b$ vertices.
Indeed, we can divide the set of regular nodes $\Nd^G$
into two independent sets by taking
\begin{itemize}
\item
$\Nd_1^G$ as the set containing $w_j$,
as well as $v_i$ for $\ell_i \not\ni p_j$,
\item
and $\Nd_2^G$ as the set containing $w_k$, as well as $v_i$ for
$\ell_i \not\ni p_k$.
\end{itemize}
These are independent sets, showing that $\Nreg$ is bipartite.
It is complete, since
\begin{itemize}
\item
If $\ell_i \not\ni p_j$, then $\ell_i \ni p_k$, so $\Nreg$ has an edge
joining $v_i$ and $w_j$.
\item
Similarly, 
If $\ell_i \not\ni p_k$, then $\ell_i \ni p_j$, so $\Nreg$ has an edge
joining $v_i$ and $w_k$.
\item
If $\ell_i \not\ni p_j$ and $\ell_{i'} \not\ni p_k$, then the lines
$\ell_i$ and $\ell_{i'}$ intersect in a double point, and so $v_i$
and $v_{i'}$ are joined by a string in $G$, i.e. they are joined by an
edge in $\Nreg$.
\item
There is a string joining $w_j$ and $w_k$ in $G$, the concatenation
of the strings corresponding to $w_j \in \ell_1$ and $w_k \in \ell_1$,
inducing an edge joining $w_j$ and $w_k$ in $\Nreg$.
\end{itemize}
Now, we normalize the graph $G$. What this means is, apply the procedure
described in the proof of \cite[Theorem 4.1]{Neu_plumb}.
In our situation, it suffices to only blow down, or apply zero-absorption
to vertices on the string joining $w_j$ and $w_k$.
There are two possible outcomes of this operation:
\begin{enumerate}[label=(\alph*)]
\item
the vertices $w_j,w_k$ remain vertices in $\Gn$, possibly with
a different Euler number, joined by a string consisting of vertices
with Euler number $\leq -2$,
\item \label{it:bad_abs}
the last operation in the sequence is a zero-absorption, making
one vertex out of $w_j,w_k$ and a zero-vertex between them.
\end{enumerate}
In order to finish the proof, it suffices to eliminate option \ref{it:bad_abs}.
If $C_j$ and $C_k$ are Seifert fibers in the pieces
corresponding to $w_j$ and $w_k$, then option \ref{it:bad_abs}
is equivalent to $C_j$ and $\pm C_k$ representing the same homology class
in the union of pieces in $\partial F_\A$ corresponding to the string.
This union is homeomorphic to $(S^1)^2 \times I$.
Let $H \cong \Z^2$ be its first integral homology group.
We will now show
that the fundamental classes $[C_j],[C_k]\in H$ are linearly independent.
It suffices to find a linear map $H \to \Z^2$ mapping $[C_j]$ and $[C_k]$
to linearly independent elements.
Denote by $u_{-r}, u_{-r-1},\ldots,u_{-1},u_0 = v_i, u_1,\ldots,u_s$
the vertices along the string joining $w_j$ and $w_k$ in $G$,
and name the vertices along it and its Euler numbers
as in \cref{fig:double_bamboo}. Set also $u_{-r-1} = w_j$ and
$u_{s+1} = w_k$.

\begin{figure}[ht]
\begin{center}
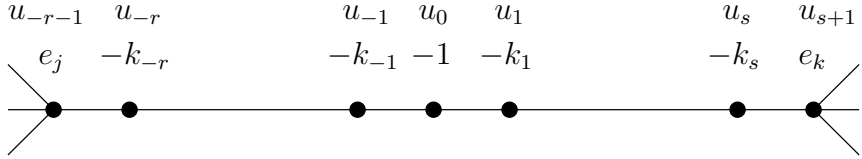
\caption{We show that this bamboo, along with the end vertices
$w_j$ and $w_k$, does not collapse to just one vertex by Neumann's
normalization algorithm.}
\label{fig:double_bamboo}
\end{center}
\end{figure}

Let $C_l$ be a Seifert fiber in the piece corresponding to $u_l$
for $-r-1 \leq l \leq s+1$. We then have linear relations in $H$
\[
  -[C_{l-1}] + k_l [C_l] - [C_{l+1}], \qquad
  -r \leq l \leq s.
\] 
Furthermore, $[C_{l-1}]$ and $[C_l]$ form a basis for any $-r \leq l \leq s+1$.
Thus, there exists a unique linear isomorphism $\psi:H\to \Z^2$ satisfying
\[
  \psi\left([C_{-1}]\right) =
  \left(
  \begin{matrix}
    0 \\ 1
  \end{matrix}
  \right),\qquad
  \psi\left([C_0]\right) =
  \left(
  \begin{matrix}
    1 \\ 1
  \end{matrix}
  \right),\qquad
  \psi\left([C_1]\right) =
  \left(
  \begin{matrix}
    1 \\ 0
  \end{matrix}
  \right).
\]
Define integers $\xi_i, \eta_i$ by
\[
  \psi\left([C_i]\right) =
  \left(
  \begin{matrix}
    \xi_i \\ \eta_i
  \end{matrix}
  \right).
\]
Then, the theory of negative continued fractions
(see e.g. \cite[III.5]{BHPVdV} with $\mu_i = \xi_i-\eta_i$,
$\nu_i = \eta_i$ for $i\geq 0$) provides
\[
  (d-n_k)(\xi_{s+1} - \eta_{s+1}) = d\xi_{s+1}, \qquad
  \gcd(\xi_i,\eta_i) = 1,
\]
and so, since $d = a+b-1$ and $n_k = a$, we find
\[
  \psi\left([C_{s+1}]\right)
  = \frac{1}{\gcd(a-1,b)}
  \left(
  \begin{matrix}
    a-1 \\ -b
  \end{matrix}
  \right).
\]
In a similar way, we find
\[
  \psi\left([C_{-r-1}]\right)
  = \frac{1}{\gcd(b-1,a)}
  \left(
  \begin{matrix}
    -a \\ b-1
  \end{matrix}
  \right).
\]
As a result, the vectors $\psi\left([C_{s+1}]\right)$
and $\psi\left([C_{-r-1}]\right)$ are linearly independent, since
\[
  \det
  \left(
  \begin{matrix}
  a-1 & -a  \\
  -b  & b-1
  \end{matrix}
  \right)
  =
  -a-b+1
  =
  -d
  \neq
  0.\qedhere
\]
\end{proof}

\begin{ex}
In \cref{fig:D33} we see the graphs $G$ and $\Gn$ in the case
of the smallest double connected pencil $\DPen_{3,3}$.
All vertices have genus zero, black dots are vertices
with Euler number $-1$, red dots are special nodes with
Euler number $-5$, black lines are regular edges, and green lines
are strings with Euler numbers $-2,-3$, from left to right.

\begin{figure}[ht]
\begin{center}
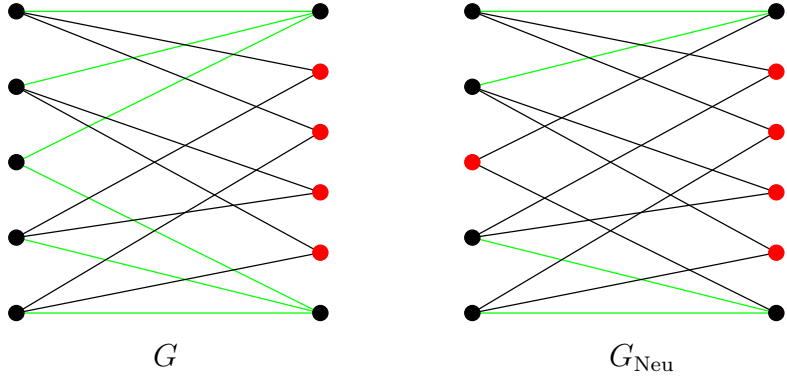
\caption{The smallest double pencil, $\DPen_{3,3}$.
As topological graphs, they are complete bipartite graphs on $3$ and $3$
vertices.}
\label{fig:D33}
\end{center}
\end{figure}

\end{ex}

\section{General case: non-exceptional arrangements} \label{s:nonexc}

We conclude by addressing the general case. In order to do so, we need to detail how to recover the intersect poset of a non-exceptional arrangement $\A$ and also how to distinguish it from the one of a double connected pencil.

\begin{lem} \label{lem:nonexc}
If $\A$ is non-exceptional,
and $\Gn$ is a plumbing graph for $\partial F_\A$ in normal form,
then $\Gn$ determines the poset
associated with $\A$.
\end{lem}
\begin{proof}
By lemma \ref{lem:double_normal}, we can assume that $\Gn = G$.
It is then clear that the special node graph $\Nsp$ is connected and bipartite,
having maximal independent sets %
$\V_1$ consisting of $v_i$ for $\ell_i \in \A$,
and $\V_2$ consisting of $w_j$ for intersection points $p_j$.
To see this, it is sufficient to see that any special node in $\Gn$
corresponds to a double point. The only other way that a special node
could arise is if there is an intersection point $p_j$ such that
\begin{itemize}
\item
we have $n_j > 2$ and and $g_j = 0$,
\item
if $\ell_i \ni p_j$, then the corresponding bamboo joining
$v_i$ and $w_j$ consists of a single vertex with Euler number $-d' \leq 2$.
\end{itemize}
Since $n_j - 2 \neq 0$, and $g_j = 0$, we have $c_j-1 = 0$.
Therefore we get
\[
  d' = \frac{d}{n_j}, \qquad \gcd(d,n_j) = 1,
\]
and so $n_j = 1$, which is impossible.
As a result of this, we see that
\begin{itemize}
\item
$(\Nsp,\V_1)$ corresponds to $P_\A$,
\item
$g_i = 0$ and $e_i = -1$ for all nodes $v_i \in \V_1$,
\item
any special node is in $\V_2$,
\item
if $v \in \V_1$ and $w \in \V_2$ are joined by a chain, and $\delta_w \geq 3$,
then $v = v_i$ for some $\ell_i \in \A$, and $w = w_i$ for some intersection
point $p_j$ with $n_j \geq 3$. The Euler numbers
$-h_1,\ldots,-h_r$ on the chain then satisfy
\[
  [h_1,\ldots,h_r] = \frac{d}{n_j} = \frac{|\V_1|}{\delta_w}.
\]
\end{itemize}
We must show that $\V_1$ is the unique maximal independent set in $N$
satisfying these conditions, i.e. that the
above conditions are not satisfied if we switch the roles of
$\V_1$ and $\V_2$.
Assume the contrary, in order to derive a contradiction.
Our assumptions thus imply
\begin{itemize}
\item
there are no special nodes, and so no double intersection points in $\A$,
\item
$g_j = 0$ and $e_j = -1$ for all intersection points $j$,
\item
with $v \in \V_1$, $w\in \V_2$ joined by a string with Euler numbers
$h_1,\ldots,h_r$ as above, we have
\begin{equation} \label{eq:reverse_h}
  [h_r, \ldots, h_1] = \frac{|\V_2|}{\delta_v}.
\end{equation}
\end{itemize}
The combined conditions $g_j = 0$ and $n_j > 2$ imply
$c_j = 1$ for all $j$. We therefore have $m_j = n_j$.
Since $e_j = -1$, we have $\eNSz_j = n_j - 1$, but we also have
$\eNSz_j = m_j'$ by Lemma \ref{lem:eNSz}. Therefore,
\begin{equation} \label{eq:mjpmjmo}
  m_j' = m_j - 1.
\end{equation}
With multiplicities $m_1, \ldots, m_s$ as
in \cref{fig:strn}, and setting $m_0 = -1$ and $m_{s+1} = m_j$,
we have equations
\[
  -m_{l-1} + k_l m_l - m_{l+1} = 0,\qquad
  l = 1,\ldots,s,
\]
which implies that we have an increasing sequence of integers
\[
  0
  = m_1 - m_0
    \leq m_2-m_1
    \leq m_3-m_2
    \leq \cdots
    \leq m_{s+1} - m_s = 1.
\]
Here, the last equality is \cref{eq:mjpmjmo}.
As a result, there exists a unique $t$ among $1,\ldots,s$ such that
\[
  m_l - m_{l-1}
  =
\begin{cases}
  0 & l \leq t,\\
  1 & l > t.\\
\end{cases}
\]
This implies that all the numbers $k_1, \ldots, k_s$ are $2$, with the sole
exception of $k_t = 3$. We therefore get
\[
  [k_1, \ldots, k_s] = [2^{t-1}, 3, 2^{s-t}],\qquad
  [h_1, \ldots, h_r] = [h_1,h_2] = [t+1, s-t+2].
\]
Furthermore, since $c_j = 1$, and
\[
  \frac{d}{n_j} = [h_1,h_2] = h_1 - \frac{1}{h_2} = \frac{h_1h_2 - 1}{h_2},
\]
we have $h_2 = n_j$ and $h_1h_2 - 1 = d$, i.e.
\begin{equation} \label{eq:theotherone}
  h_1 = \frac{d+1}{n_j}.
\end{equation}
Similarly, using \cref{eq:reverse_h}, we find that if the string
joins $v_i$ and $w_j$, then
\[
  h_1 = n_i,\qquad
  h_2 = \frac{|\V_2|+1}{n_i}.
\]
In particular, the numbers $h_1,h_2$ only depend on the vertex $v_i$, i.e.
any string adjacent to the same $v_i$ has the same Euler numbers $h_1, h_2$.
Since the same is true for strings adjacent to any vertex $w_j$, it
follows that all strings joining any pair $v\in \V_1$ and $w \in \V_2$
have the same Euler numbers $h_1, h_2$. In particular, all vertices
$v_i \in \V_1$ have the same valency $\bar{n}_i = h_1$, and all vertices
$w_j \in \V_2$ have the same valency $n_j = h_2$. Thus,
any line $\ell_i$ contains $\bar{n}_i$ singular points, through each of which,
$n_j = h_2$ lines pass, including $\ell_i$. Since every line in $\A$ passes
through one of these points, we find
\[
  d = \bar{n}_i (n_j - 1) + 1.
\]
At the same time, by \cref{eq:theotherone}, we find $\bar{n}_i n_j = d + 1$.
Together, these equations imply $\bar{n}_i = 2$, contradicting
$\A$ being non-exceptional.
\end{proof}

\begin{lem}\label{lem:not_bipartite}
If $\A$ is not exceptional, then the regular node graph $\Nreg$
associated with $\Gn$
is not a complete bipartite graph.
\end{lem}
\begin{proof}
By Lemma \ref{lem:double_normal}, the regular nodes of $\Gn$
correspond to precisely the lines $\ell_i$ and intersection points
$p_j$ with $n_j \geq 3$, and $\Gn$ is connected.
Denote by $\Lambda$ the set of lines,
and $\Pi$ the set of intersection points $p_j$ with $n_j \geq 3$.
If $\Nreg$ is bipartite, then there exists a unique partitioning of each set
\[
  \Lambda = \Lambda_1 \amalg \Lambda_2,\qquad
  \Pi = \Pi_1 \amalg \Pi_2
\]
corresponding to a partitioning of the regular nodes in $\Gn$.
If $\ell_{i_1} \in \Lambda_1$ and $\ell_{i_2} \in \Lambda_2$, then
these lines intersect in a double point, providing an edge in $\Nreg$,
whereas if $\ell_i,\ell_{i'} \in \Lambda_1$, then $\ell_i$ and $\ell_{i'}$
intersect in a point of multiplicity $\geq 3$.
If $\Nreg$ is a \emph{complete} bipartite graph, then there is an
edge joining any pair of vertices in $\Pi_1$ and $\Pi_2$. But, since
there is no line containing only two intersection points, a vertex
in $\Pi_1$ is not joined with a vertex in $\Pi_2$, so one of these
sets is empty, let us assume that $\Pi_1 = \emptyset$.
If $\Pi_2 = \emptyset$ as well, then $\A$ is generic, i.e. has only double
interesction points. In this case, the node graph $\Nreg$ is a complete graph
on $d$ vertices, and so either not bipartite, or exceptional.

Now,
any line $\ell_i$ containing a point of multiplicity $\geq 3$
is in $\Lambda_1$, so $\Lambda_2$ consists of \emph{generic} lines,
i.e. lines which only contain double points.
Any such generic line is necessarily in $\Lambda_2$, since it intersects
any line in $\Lambda_1$ in a double point.
Two distinct generic lines cannot exist, since they would
intersect each other in a double point.
As a result, there can be at most one generic line.
If there is no generic line, then all intersection points have multiplicity
$\geq 3$, and all lines contain all points. Thus, there is either exactly
one line, $d=1$, which is not the case, since $\A$ is non-exceptional,
or there is exactly one intersection point, which is also excluded,
since the pencil is exceptional.
If there does exists a generic line $\ell_i \in \Lambda_2$,
then the arrangement $\A' = \A\setminus\{\ell_i\}$ is of the above type,
which implies that $\A$ either contains two lines, or it is a near-pencil.
But these cases are exceptional.
\end{proof}


\end{document}